\documentclass{amsart}

\usepackage{amssymb,amsmath,amsthm,color}

\usepackage[bookmarksopen=true,bookmarksnumbered=true, bookmarksopenlevel=2, colorlinks=true,linkcolor=mblue,
citecolor=mblue,urlcolor=mblue, pdfstartview=FitH]{hyperref}

\definecolor{mblue}{rgb}{0,0,.8}

\usepackage[all]{xy}

\usepackage{enumerate, longtable}

\newcommand{\N}{\mathbb N}
\newcommand{\Z}{\mathbb Z}
\newcommand{\Q}{\mathbb Q}
\newcommand{\Qbar}{\overline{\Q}}
\newcommand{\Zbar}{\overline{\Z}}
\newcommand{\F}{\mathbb F}
\newcommand{\Fbar}{\overline{\F}}

\newcommand{\C}{\mathbb C}
\newcommand{\p}{\mathfrak p}

\newcommand{\q}{\mathfrak q}

\newcommand{\OO}{\mathcal O}

\def\cross{\times}

\newcommand{\Hom}{\operatorname{Hom}}

 \newcommand{\abcd}[4]{\left(
         \begin{smallmatrix}#1&#2\\#3&#4\end{smallmatrix}\right)}

\newcommand{\stroke}[1]{[#1]}
\newcommand{\diam}[1]{\langle #1 \rangle}

\newcommand{\calS}{\mathcal{S}}
\newcommand{\TT}{\mathbb{T}}
\newcommand{\fm}{\mathfrak{m}}
\newcommand{\fP}{\mathfrak{P}}

\newcommand{\Zmod}[1]{\overline{\Z/{#1}\Z}}
\newcommand{\ilim}{\underset{\rightarrow}{\mathrm{lim}}}
\newcommand{\prlim}{\underset{\leftarrow}{\mathrm{lim}}}

\newtheorem{thm}{Theorem}
\newtheorem{lem}[thm]{Lemma}
\newtheorem{dfn}[thm]{Definition}
\newtheorem{prop}[thm]{Proposition}
\newtheorem{cor}[thm]{Corollary}
\newtheorem*{question}{Question}

\DeclareMathOperator{\GL}{GL} \DeclareMathOperator{\SL}{SL}  
\DeclareMathOperator{\Gal}{Gal}   

  \DeclareMathOperator{\Tr}{tr}
\DeclareMathOperator{\Frob}{Frob}

\DeclareMathOperator{\End}{End}

\def\dash---{\thinspace---\hskip.16667em\relax}

\def\o{{\mathcal O}}

\begin{document}

\title{On modular Galois representations\\ modulo prime powers.}
\author[Imin Chen, Ian Kiming, Gabor Wiese]{Imin Chen, Ian Kiming, Gabor Wiese}
\address[Imin Chen]{Department of Mathematics, Simon Fraser University, 8888 University Drive, Burnaby, B.C., V5A 1S6, Canada}
\email{ichen@math.sfu.ca}
\address[Ian Kiming]{Department of Mathematical Sciences, University of Copenhagen, Universitetsparken 5, DK-2100 Copenhagen \O , Denmark}
\email{kiming@math.ku.dk}
\address[Gabor Wiese]{Universit\'e du Luxembourg, Facult\'e des Sciences, de la Technologie et de la Communication, 6, rue Richard Coudenhove-Kalergi, L-1359 Luxembourg, Luxembourg}
\email{gabor.wiese@uni.lu}

\begin{abstract} We study modular Galois representations mod $p^m$. We show that there are three progressively weaker notions of modularity for a Galois representation mod $p^m$: we have named these `strongly', `weakly', and `dc-weakly' modular. Here, `dc' stands for `divided congruence' in the sense of Katz and Hida. These notions of modularity are relative to a fixed level $M$.

Using results of Hida we display a level-lowering result (`stripping-of-powers of $p$ away from the level'): A mod $p^m$ strongly modular representation of some level $Np^r$ is always dc-weakly modular of level $N$ (here, $N$ is a natural number not divisible by $p$).

We also study eigenforms mod $p^m$ corresponding to the above three notions. Assuming residual irreducibility, we utilize a theorem of Carayol to show that one can attach a Galois representation mod $p^m$ to any `dc-weak' eigenform, and hence to any eigenform mod $p^m$ in any of the three senses.

We show that the three notions of modularity coincide when $m=1$ (as well as in other, particular cases), but not in general.
\end{abstract}

\maketitle

\section{Introduction}\label{intro}

Let $p$ be a prime number, which remains fixed throughout the article. Let $N$ be a natural number not divisible by $p$. All number fields in the article are taken inside some fixed algebraic closure $\Qbar$ of~$\Q$ and we fix once and for all field embeddings $\Qbar \hookrightarrow \Qbar_p$, $\Qbar \hookrightarrow \C$, as well as a compatible isomorphism $\C \cong \Qbar_p$. Here, $\Qbar_p$ is a fixed algebraic closure of $\Q_p$. We also denote by $\Zbar_p$ the ring of integers of $\Qbar_p$.

Let $f = \sum_{n=1}^\infty a_n(f) q^n \in S_k(\Gamma_1(N p^r))$ be a normalized cuspidal eigenform for all Hecke operators $T_n$ with $n \ge 1$ (normalized means $a_1(f) = 1$). By Shimura, Deligne, and Serre, there is a continuous Galois representation
$$
\rho = \rho_{f,p} : \Gal(\Qbar/\Q) \rightarrow \GL_2(\Qbar_p)$$ attached
to $f$, which is unramified outside $Np$, and which satisfies $$\Tr
\rho(\Frob_\ell) = a_\ell(f) \text{ and } \det \rho(\Frob_\ell) =
\ell^{k-1} \chi(\ell),
$$ for primes $\ell \nmid Np$, where $\chi$ is the nebentypus of $f$.

By continuity and compactness, the Galois representation descends to a representation
$$
\rho_{f,\Lambda,p} : \Gal(\Qbar/\Q) \rightarrow \GL_2(\o_K),
$$
where $\o_K$ is the ring of integers of a finite extension $K$ of $\Q_p$. This representation depends in general on a choice of $\o_K$-lattice $\Lambda$ in~$K^2$. Let $\p = \p_K$ be the maximal ideal of $\o_K$. We wish to consider the reduction mod $\p^m$ of $\rho_{f,\Lambda,p}$. However, because of ramification, the exponent $m$ is not invariant under base extension.

For this reason, it is useful, following \cite{taixes-wiese}, to define $\gamma_K(m) := (m-1) e_{K/\Q_p}+1$, with $e_{K/\Q_p}$ the ramification index of $K/\Q_p$. This definition is made precisely so that the natural maps below yield injections of rings, i.e.\ ring extensions of $\Z/p^m\Z$,
$$
\Z/p^m\Z \hookrightarrow \o_K/\p_K^{\gamma_K(m)} \hookrightarrow \o_L/\fP_L^{\gamma_L(m)}
$$
for any finite extension $L/K$ (with $\fP_L$ the prime of $L$ over $\p_K$ in $K$). We can thus form the ring
$$
\Zmod{p^m} := \ilim_K \o_K/\p_K^{\gamma_K(m)},
$$
which we also consider as a topological ring with the discrete topology. When we speak of $\alpha \pmod {p^m}$ for $\alpha \in \Zbar_p$, we mean its image in $\Zmod{p^m}$. In particular, for $\alpha, \beta \in \overline{\Z}_p$, we define $\alpha \equiv \beta \pmod{p^m}$ as an equality in $\Zmod{p^m}$, or equivalently, by $\alpha - \beta \in \p_K^{\gamma_K(m)}$, where $K/\Q_p$ is any finite extension containing $\alpha$ and $\beta$.

In this spirit, we define the reductions
$$
\rho_{f,\Lambda,p,m} : \Gal(\Qbar/\Q) \rightarrow \GL_2(\Zmod{p^m})
$$
for any $m \in \N$. The representation $\rho_{f,\Lambda,p,m}$ has the property:
$$
\Tr \rho_{f,\Lambda,p,m} (\Frob_\ell) = (a_\ell(f) \bmod p^m)
\leqno{(\ast)}
$$
for all primes $\ell\nmid Np$.
\smallskip

From \cite[Th\'eor\`eme 1]{carayol} and the Chebotarev density theorem, a continuous Galois representation $\rho_{p,m} : \Gal(\Qbar/\Q) \rightarrow \GL_2(\Zmod{p^m})$ is determined uniquely up to isomorphism by $\Tr \rho_{p,m}(\Frob_\ell)$ for almost all (i.e., all but finitely many) primes $\ell$, assuming the residual representation is absolutely irreducible. It follows that if there is one choice of $\o_K$-lattice $\Lambda$ as above such that
$\rho_{f,\Lambda,p,1}$ is absolutely irreducible, then $\rho_{f,p,m} = \rho_{f,\Lambda,p,m}$ is determined uniquely up to isomorphism. In such a case, we say $\rho_{f,p,m}$ is the mod $p^m$ Galois representation attached to $f$.
\medskip

Let $M\in\N$. The $\C$-vector spaces $S = S_k(\Gamma_1(M))$ and $S = S^b(\Gamma_1(M)): = \oplus_{i=1}^b S_i(\Gamma_1(M))$ have integral structures, so it is possible to define (arithmetically) the $A$-module of cusp forms in $S$ with coefficients in a ring $A$, which we denote by $S(A)$ (cf.\ Section~\ref{intstr}). The spaces $S(A)$ have actions by the Hecke operators $T_n$ for all $n \ge 1$. We point out that every $f \in S(\Zmod{p^m})$ can be obtained as the reduction of some $\tilde{f} \in S(\o_K)$ for some number field (or $p$-adic field)~$K$. However, for $m > 1$, if $f$ is an eigenform for the Hecke operators $T_n$ over $\Zmod{p^m}$ (for all $n\ge 1$ coprime to some fixed positive integer $D$), the lift $\tilde{f}$ cannot be chosen as an eigenform (for the same Hecke operators), in general.

We introduce the following three progressively weaker notions of eigenforms mod $p^m$.

\begin{dfn}\label{dfn:dc}
\begin{itemize} \item We say that $f \in S_k(\Gamma_1(M))(\Zmod{p^m})$ is a {\em strong Hecke eigenform (of level $M$ and weight $k$ over $\Zmod{p^m}$}) if there is an element $\tilde{f} \in S_k(\Gamma_1(M))(\Zbar_p)$ that reduces to $f$ and a positive integer $D$ such that
$$ T_n\tilde{f} = a_n(\tilde{f})\cdot \tilde{f} \textnormal{ and } a_1(\tilde{f})=1$$
for all $n\ge 1$ coprime with $D$, where $a_n(\tilde{f})$ is the $n$-th coefficient in the $q$-expansion of~$\tilde{f}$ at~$\infty$.
\item We say $f \in S_k(\Gamma_1(M))(\Zmod{p^m})$ is a {\em weak Hecke eigenform (of level $M$ and weight $k$ over $\Zmod{p^m}$}) if
$$ T_n f = f(T_n) \cdot f \textnormal{ and } f(T_1)=1$$
for all $n \ge 1$ coprime to some positive integer~$D$. The piece of notation $f(T_n)$ is defined in~Section~\ref{intstr} and represents the $n$-th coefficient of the formal $q$-expansion of~$f$.
\item  We say $f \in S^b(\Gamma_1(M))(\Zmod{p^m})$ is a {\em dc-weak Hecke eigenform (of level $M$ and in weights $\le b$ over $\Zmod{p^m}$)} if
$$ T_n f = f(T_n) \cdot f \textnormal{ and } f(T_1)=1$$
for all $n \ge 1$ coprime to some positive integer~$D$. Here, dc stands for `divided congruence' for reasons that will be explained in detail below.
\end{itemize}
\end{dfn}

We point out that above (and also in Definition \ref{dfn:eigenform}), we define eigenforms as `eigenforms away from finitely many primes' because this extra flexibility is very useful in our applications. We will specify this implicit parameter $D$ when necessary. In that case, we will speak of an {\em eigenform away from~$D$}.

There are natural maps
\begin{align*}
 & S_k(\Gamma_1(M))(\overline{\Z}_p) \rightarrow S_k(\Gamma_1(M))(\Zmod{p^m})), \\
 & S_k(\Gamma_1(M))(\Zmod{p^m})) \hookrightarrow  S^b(\Gamma_1(M))(\Zmod{p^m}), k \le b \\
 & S^b(\Gamma_1(M))(\Zmod{p^m}) \hookrightarrow \ilim_{c \ge 1} S^c(\Gamma_1(M))(\Zmod{p^m}).
\end{align*}
It follows that a strong eigenform of level $M$ and some weight is a weak eigenform of the same level and weight, and that a weak eigenform of level~$M$ and some weight gives rise to a dc-weak eigenform of the same level. Furthermore, if we regard our eigenforms inside the last direct limit, we can make sense of when two eigenforms (of the various kinds) are the same.

We derive in this article from Katz-Hida theory (\cite{katz}, \cite{hida1}, \cite{hida2}) that a dc-weak form of some level~$Np^r$ is dc-weak of level~$N$ (recall $p \nmid N$), under some mild technical restrictions. More precisely, we prove:

\begin{prop}\label{prop:hatada} Suppose that $p\ge 5$. Let $f \in S^b(\Gamma_1(Np^r))(\Zmod{p^m})$. Then there is $c \in \N$ and an element $g \in S^c(\Gamma_1(N))(\Zmod{p^m})$ such that
$$
f(q)= g(q) \in \Zmod{p^m}\,[[q]],
$$
where $f(q)$ and $g(q)$ denote the $q$-expansions of $f$ and $g$.
\end{prop}

\medskip

As we show in Lemma~\ref{lem:ds}, when $m=1$, a dc-weak $f$ eigenform over $\Zmod{p^m}$ is in fact a strong eigenform. This means that we have a mod $p$ Galois representation $\rho_{f,p,1}$ attached to $f$. Using results of Carayol \cite{carayol}, we show more generally that it is possible to attach mod $p^m$ Galois representations to dc-weak eigenforms of level $M$ over $\Zmod{p^m}$, whenever the residual representation is absolutely irreducible. More general results involving infinite-dimensional completed Hecke algebras are due to Mazur, Gouv\^ea (see specifically Corollary III.5.8 of \cite{gouvea}), Hida, and Wiles, but we provide a self-contained proof based on \cite{carayol}.

\begin{thm}\label{main-galois}
Let $f$ be a dc-weak eigenform of level $M$ over $\Zmod{p^m}$. Assume that the residual representation $\rho_{f,p,1}$ is absolutely irreducible. Then there is a continuous Galois representation
$$ \rho_{f,p,m}: \Gal(\Qbar/\Q) \to \GL_2(\Zmod{p^m})
$$
unramified outside $M p$ such that for almost all primes $\ell \nmid M p$ we have
$$
\Tr(\rho_{f,p,m}(\Frob_\ell)) = f(T_\ell),
$$
where $f(T_\ell) \in \Zmod{p^m}$ is the eigenvalue of the operator~$T_\ell$
(see Section~\ref{sec:mf}).
\end{thm}

One can be more precise: if the normalized dc-weak eigenform in the theorem is an eigenform for all $T_n$ with $n$ coprime to the positive integer~$D$, then the final equality holds for all primes $\ell \nmid DMp$.

Having Galois representations attached to dc-weak eigenforms (and hence also to weak eigenforms), we can consider modularity questions. Towards this aim, we introduce further terminology corresponding to the above three notions of `eigenform mod $p^m$':

\begin{dfn}\label{dfn:dc-galois}
Given a Galois representation
$$
\rho_{p,m} : \Gal(\Qbar/\Q) \rightarrow \GL_2(\Zmod{p^m})
$$
which we assume to be residually absolutely irreducible, we say that $\rho_{p,m}$
\begin{itemize} \item {\it strongly arises from $\Gamma_1(M)$} if $\rho_{p,m}$ is isomorphic to $\rho_{f,p,m}$ for some strong Hecke eigenform $f$. As discussed above,
this is equivalent to
\begin{equation*}
  \Tr \rho_{p,m}(\Frob_\ell) = (a_\ell(\tilde{f}) \bmod p^m)
\end{equation*}
almost all primes $\ell \nmid M p$, where $\tilde{f}$ is a normalized eigenform over~$\Zbar_p$ lifting~$f$.
\item We say $\rho_{p,m}$ {\em weakly arises from $\Gamma_1(M)$} if $\rho_{p,m}$ is isomorphic to $\rho_{f,p,m}$ for some weak Hecke eigenform of level $M$ and some weight $k$ over $\Zmod{p^m}$.
\item We say that $\rho_{p,m}$ {\em dc-weakly arises from $\Gamma_1(M)$} if $\rho_{p,m}$ is isomorphic to $\rho_{f,p,m}$ for some dc-weak Hecke eigenform of level $M$ over $\Zmod{p^m}$.
\end{itemize}
\end{dfn}

One of the motivations behind the present work is building a framework for understanding such mod $p^m$ representations. We consider it conceptually very important to establish a precise link between mod $p^m$ Galois representations, especially those that arise as reductions of $p$-adic ones, and modular forms mod $p^m$. Such a link would have many important consequences of a theoretical and computational nature.

Results about mod $p$ Galois representations and their modularity have proven very useful in applications to Diophantine equations - such as in the proof of Fermat's Last Theorem, and many other families of generalized Fermat equations. In \cite{dahmen-yazdani}, a Ribet-type level lowering result for mod $p^m$ Galois representations is proven and applied to resolve some new cases of generalized Fermat equations.

In the present paper we treat the question of `stripping powers of $p$ away from the level':

\begin{question} If $\rho_{p,m}$ strongly arises from $\Gamma_1(Np^r)$, does it strongly arise from $\Gamma_1(N)$?
\end{question}

As is well-known, mod $p$ Galois representations $\rho_{f,p,1}$ always strongly arise from $\Gamma_1(N)$, cf.\ Ribet \cite[Theorem 2.1]{ribet} for $p\ge 3$, and Hatada \cite[Theorem 2]{hatada} for $p\ge 2$. However, when the nebentypus has a non-trivial component of $p$-power conductor and order, the representations $\rho_{f,p,m}$ do not in general even weakly arise from $\Gamma_1(N)$ if $m \ge 2$, as we show in Section \ref{relations}.

Let us remark that in the deduction that mod $p$ Galois representations in level $Np^r$ always strongly arise from level $N$ (rather than just weakly arise or dc-weakly arise) one uses two facts: the validity of the Deligne--Serre lifting lemma for mod $p$ representations, and the absence of constraints on the determinant of mod $p$ Galois representations arising from strong Hecke eigenforms. Neither of these facts are true in general for mod $p^m$ representations.

However, we derive from Katz-Hida theory, via Proposition~\ref{prop:hatada}, the following weaker version mod $p^m$ of the above level-lowering result.

\begin{thm}\label{thm:main} Let $f$ be a dc-weak eigenform of level $Np^r$ over $\Zmod{p^m}$. Assume that the residual representation $\rho_{f,p,1}$ is absolutely irreducible. Also assume $p \ge 5$.

Then the representation $\rho_{f,p,m}$ dc-weakly arises from $\Gamma_1(N)$.
\end{thm}

\begin{proof}
This follows by a combination of Proposition~\ref{prop:hatada} and Theorem~\ref{main-galois}: By definition the form $f$ is a normalized eigenform for all $T_n$ with $(n,D) = 1$. Pick a form $g$ at level $N$ according to Proposition~\ref{prop:hatada}. Enlarging $D$ if necessary so as to have $p\mid D$, we can be sure that $g$ is also a normalized eigenform for all $T_n$ with $(n,D) = 1$. As the Galois representation attached to $g$ by Theorem~\ref{main-galois} is isomorphic to $\rho_{f,p,m}$, the desired follows.
\end{proof}

We stress the following particular consequence of the theorem. If $f \in S_k(\Gamma_1(Np^r))$ is a normalized eigenform, then there exists a number field~$K$ and a $g \in S^b(N)(\o_K)$ (which cannot be taken to be an eigenform, in general) such that $g \pmod{p^m}$ is a dc-weak eigenform and its attached Galois representation $\rho_{g,p,m}$ is isomorphic to~$\rho_{f,p,m}$. We do not treat in this paper the more difficult question of weight information concerning such a $g$.

We note that a result of Hatada, \cite[Theorem 1]{hatada}, has a consequence that can be interpreted as the following statement, showing that the need for divided congruence forms only appears when the nebentypus has a non-trivial component of $p$-power conductor and order.

\begin{thm}[Hatada] \label{thm:hatada} Let $f$ be a strong eigenform of level $Np^r$ and weight $k$ over $\Zmod{p^m}$ such that $\diam{\ell} f = \chi(\ell) f$, where $\chi$ has no non-trivial component of $p$-power conductor and order. Then the representation $\rho_{f,p,m}$ weakly arises from $\Gamma_1(N)$.
\end{thm}

The paper is organized as follows. In Section \ref{sec:mf} we provide background information on integral structures on spaces of modular forms, Hecke algebras, and divided congruence forms. In Section \ref{sec:galois} we construct Galois representations attached to dc-weak eigenforms. In Section \ref{sec:stripping} we give a proof of our level-lowering result in the mod $p^m$ setting. Finally, in Section \ref{relations} we make a number of observations on the relations between the three notions `strong', `weak', and `dc-weak' of eigenforms mod $p^m$.

\subsection*{Acknowlegdements}

The authors would like to thank the referee for his/her careful reading and the very good suggestions concerning presentation.

I.~K.\ acknowledges support from The Danish Council for Independent Research.

I.~C.\ acknowledges support from NSERC.

I.~C. and G.~W.\ would like to thank the University of Copenhagen, where part of this research was done, for its hospitality.

G.~W.\ acknowledges partial support by the DFG Priority Program 1489.

\section{Modular forms and Hecke algebras}\label{sec:mf} All material in this section is well-known. We present it here in a concise form.

\subsection{\texorpdfstring{$q$}{q}-expansions}

Let $\calS = \bigoplus_{k \in \N} S_k(\Gamma_1(M))$ be the $\C$-vector space of all cusp forms of any positive weight at a fixed level $M$. Let each Hecke operator $T_n$ act on $\calS$ via the diagonal action. We will be considering finite-dimensional subspaces $S \subseteq \calS$ of the following type:
$$
S = S^b(\Gamma_1(M)) := \bigoplus_{k=1}^b S_k(\Gamma_1(M))
$$
for any $b \in \N$, $M \ge 1$.
Such a subspace $S$ is stabilized by $T_n$ for all $n \ge 1$.

For $f \in \calS$, let $f(q) \in \C[[q]]$ denote its $q$-expansion. We denote the {\em $q$-expansion map} on $S \subset \calS$ by
$$
\Phi_{S} : S \to \C[[q]], f \mapsto f(q) = \sum_{n \ge 1} a_n(f) q^n.
$$

\begin{prop}\label{prop:inj}
Fix $M \in \N$ and $b\in \N$. Let $S := S^b(\Gamma_1(M))$. Then $\Phi_S$ is injective.
\end{prop}

\begin{proof} Let $f_k \in S_k(\Gamma_1(M))$, for $k=1,\ldots, b$ be such that $\sum_{k=1}^b f_k(q) = 0$. The function $\sum_{k=1}^b f_k$ is holomorphic and $1$-periodic and hence uniquely determined by its Fourier series. Hence, $\sum_{k=1}^b f_k = 0$ and it then follows from \cite{miyake}, Lemma 2.1.1, that we have $f_k=0$ for each $k$.
\end{proof}

\subsection{Hecke algebras}

Let $R$ be a subring of~$\C$. Let
$$
\TT_R(S) \subseteq \End_\C(S)
$$
be the {\em $R$-Hecke algebra} associated to $S \subset \calS$, defined as the $R$-algebra generated by Hecke operators $T_n$, $n\ge 1$.

\begin{lem}\label{lem:pairing} Let $S := S^b(\Gamma_1(M))$.
\begin{enumerate}[(a)]
\item Let $f \in S$. Then $a_1(T_n f) = a_n(f)$ for all~$n \ge 1$.
\item Let $R \subseteq \C$ be a subring. Then the pairing of $R$-modules
$$
\TT_R(S) \times S \to \C, \;\;\; (T,f) \mapsto a_1(Tf)
$$
is non-degenerate.
\end{enumerate}
\end{lem}

\begin{proof} (a) follows, since the equation $a_1(T_n f) = a_n(f)$ is true on every summand of~$S$, hence also in the sum.

(b) Let first $f \in S$ be given. If $a_1(Tf) = 0$ for all $T \in \TT_R(S)$, then, in particular, $a_1(T_n f) = a_n(f) = 0$ for all~$n$, whence~$f$ is zero by the injectivity of the $q$-expansion map. Let now $T \in \TT_R(S)$ be given. If $a_1(Tf) = 0$ for all $f \in S$, then, in particular, $0=a_1(T (T_n f))=a_n(Tf)$ for all $f$ and all~$n$. Thus, as before we conclude that $Tf$ is zero for all~$f$, which by definition means $T = 0$.
\end{proof}

\subsection{Integral structures and cusp forms with coefficients in rings}\label{intstr} As is well-known, the spaces $S_k(\Gamma_1(M))$ have integral structures in the sense that $S_k(\Gamma_1(M))$ contains a full lattice which is stable under the Hecke operators $T_n$ for all $n \ge 1$ (cf.\ for instance \cite[Proposition 2.7]{deligne-serre}). It follows the space $S = S^b(\Gamma_1(M)) := \oplus_{k=1}^b S_k(\Gamma_1(M))$ also contains a full lattice stable under the Hecke operators $T_n$ for all $n \ge 1$.

Thus, $\TT_\Z(S)$ sits inside an integer matrix ring, and  $\TT_\C(S)$ sits inside the corresponding complex matrix ring. This implies that $\TT_\Z(S)$ is free and finite over $\Z$. Furthermore, the natural homomorphism $\TT_\Z(S) \otimes \C \rightarrow \TT_\C(S)$ is injective, so as $T_n \otimes 1$ is sent to $T_n$ and these generate $\TT_\C(S)$ over $\C$, this is an isomorphism
$$
\TT_\Z(S) \otimes \C \cong \TT_\C(S).
$$

Hence we also see that the map
$$
\alpha: \Hom_\C (\TT_\C(S) ,\C ) \longrightarrow \Hom_\C (\TT_\Z(S) \otimes \C ,\C) \longrightarrow \Hom_\Z (\TT_\Z(S),\C)
$$
coming from $\TT_\Z(S) \rightarrow \TT_\Z(S) \otimes \C \rightarrow \TT_\C(S)$ is an isomorphism (the last arrow is always an isomorphism).

Now, $S \cong \Hom_\C (\TT_\C(S) ,\C )$ as we have a non-degenerate pairing between these two complex vector spaces, cf.\ Lemma \ref{lem:pairing}. Explicitly, we obtain an isomorphism
$$
\beta: \Hom_\C (\TT_\C(S),\C) \longrightarrow S
$$
by mapping $\phi\in \Hom_\C (\TT_\C(S),\C)$ to $\sum_n \phi(T_n) q^n$ ($q=e^{2\pi i z}$). By the above isomorphisms, it follows that the map
$$
\Psi_S := \beta \circ \alpha^{-1} : ~ \Hom_\Z (\TT_\Z(S),\C) \longrightarrow S,
$$
which satisfies
$$
\Psi_S(\phi) = \sum_n \phi(T_n) q^n,
$$
is well defined and is an isomorphism.

\begin{dfn}
Let $A$ be any commutative ring~$A$. We let
$$
S(A) := \Hom_\Z(\TT_\Z(S),A) \;\;\; \textnormal{($\Z$-linear homomorphisms)},
$$

which we call the {\em cusp forms in $S$ with coefficients in~$A$}.
\end{dfn}
The $A$-module $S(A)$ is equipped with a natural action of $\TT_\Z(S)$ given by
$$(T.f)(T') = f(T T').$$
Note that $S(\C) \cong S$.
We remark that for any ring~$A$ and any $1 \le k \le b$, the map
$$
S_k(\Gamma_1(M))(A) \to S^b(\Gamma_1(M))(A), \;\;\; f \mapsto f \circ \pi,
$$
is an injective $A$-module homomorphism, where $\pi$ is the surjective ring homomorphism
$$
\TT_\Z(S^b(\Gamma_1(M))) \to \TT_\Z(S_k(\Gamma_1(M)))
$$
defined by restricting Hecke operators.

We mention that for $k \ge 2$ and $M \ge 5$ invertible in~$A$,
$S_k(\Gamma_1(M))(A)$ coincides with the corresponding $A$-module of Katz modular forms
(see e.g.\ \cite{diamond-im}, Theorem~12.3.2).

\begin{dfn}\label{dfn:eigenform}
We say that a cusp form $f \in S(A)$ is a {\em normalized Hecke eigenform} if there is a positive integer $D$ such that
$$ T_n f = f(T_n) \cdot f \textnormal{ and } f(T_1)=1$$
for all integers $n$ coprime to~$D$.
\end{dfn}

The above notion of normalized Hecke eigenform is consistent with Definition~\ref{dfn:dc}. As before, if the parameter~$D$ is specified, then we will speak of an {\em eigenform away from~$D$}.

The chosen isomorphism $\C \cong \Qbar_p$ identifies $S \cong S(\C)$ with $S(\Qbar_p)$. Hence, via the isomorphism $S(\C) \cong S$ normalized holomorphic eigenforms in $S$ (in the usual sense, for almost all Hecke operators) are precisely the normalized eigenforms in $S(\Qbar_p)$. In the following we will identify forms in $S$ with elements in either of the two spaces $S(\C)$ and $S(\Qbar_p)$.

Let $R \subseteq \C$ be a subring.
For a positive integer~$D$, let $\TT^{(D)}_R(S)$ be the $R$-subalgebra of $\TT_R(S)$ generated by those Hecke operators $T_n$ for which $n$ and $D$ are coprime.

\begin{lem}\label{lem:eigenforms} Let $A$ be a ring, $f \in S(A)$ a normalized Hecke eigenform and $D$ as in Definition~\ref{dfn:eigenform}.

Then the restriction of $f$ to $\TT_\Z^{(D)}$ is a ring homomorphism.
\end{lem}

\begin{proof}
The claim immediately follows from the equation
$$
f (TT') = (T.f)(T') = f(T) f(T')
$$
with $T,T' \in \TT_\Z^{(D)}(S)$.
\end{proof}

\subsection{Integral structures for Hecke algebras, base change and lifting}
Using the integral structure on~$S$ we can also equip Hecke algebras with an integral structure in the following sense.

\begin{lem}\label{lem:intstr} Fix $M,b \in \N$ and let $S := \bigoplus_{k=1}^b S_k(\Gamma_1(M))$. Let $R \subseteq \C$ be a subring.
\begin{enumerate}[(a)]
\item The $R$-Hecke algebra $\TT_R(S)$ is free as an $R$-module of rank equal to $\dim_\C S$, in particular, $\TT_\Z(S)$ is a free $\Z$-module of that rank.
\item $\TT_R(S) \cong \TT_\Z(S) \otimes_\Z R$.
\end{enumerate}
\end{lem}

\begin{proof} We know that the natural map $\TT_\Z(S)\otimes\C \rightarrow \TT_\C(S)$ is an isomorphism and that $\TT_\Z(S)$ is free and finite over $\Z$. It follows that $\TT_\Z(S)$ is a lattice of full rank in $\TT_\C(S)$. That rank is the $\C$-dimension of $\TT_\C(S)$, i.e., of $\Hom_\C (\TT_\C(S) , \C) \cong S$. Hence we have (a) for $R=\Z$.

It follows immediately that $\TT_\Z(S) \otimes_\Z R$ is a free $R$-module of the same rank, which surjects onto $\TT_R(S)$. Now, $\TT_\Z(S)$ has a $\Z$-basis which remains linearly independent over~$\C$. Thus, it also remains linearly independent over~$R$, and its $R$-span is by definition $\TT_R(S)$, which is thus a free $R$-module of the same rank and is isomorphic to $\TT_\Z(S)\otimes_\Z R$.
\end{proof}

We further obtain that cusp forms in $S$ with coefficients behave well with respect to arbitrary base change:

\begin{lem}\label{lem:basechange} Fix $M,b \in \N$ and let $S := \bigoplus_{k=1}^b S_k(\Gamma_1(M))$. Let $A \to B$ be a ring homomorphism. Then $S(A) \otimes_A B \cong S(B)$.
\end{lem}

\begin{proof}
By Lemma~\ref{lem:intstr} we know that $\TT_\Z(S)$ is a free $\Z$-module of some finite rank~$d$. Hence:
$S(A) \otimes_A B
= \Hom_\Z(\TT_\Z(S),A) \otimes_A B
\cong \Hom_\Z(\Z^d,A) \otimes_A B
\cong A^d \otimes_A B
\cong B^d
\cong \Hom_\Z(\Z^d,B)
\cong \Hom_\Z(\TT_\Z(S),B)
= S(B)$.
\end{proof}

For the sake of completeness we also record the following simple lifting property.

\begin{lem}\label{lem:lifting} Fix $M,b \in \N$ and let $S := \bigoplus_{k=1}^b S_k(\Gamma_1(M))$.

Let $f \in S(\Zmod{p^m})$. Then there is a number field (and hence there is also a $p$-adic field)~$K$ and $\tilde{f} \in S(\o_K)$ such that $\tilde{f} \equiv f \pmod{p^m}$, in the sense that $\tilde{f}(T_n) \equiv f(T_n) \pmod{p^m}$ for all $n \in \N$.
\end{lem}

\begin{proof}
As $\TT_\Z(S)$ is a free $\Z$-module of finite rank (Lemma~\ref{lem:intstr}), it is a projective $\Z$-module. Moreover, the image of the homomorphism (of abelian groups) $f: \TT_\Z(S) \to \Zmod{p^m}$ lies in $\o_K/\p_K^{\gamma_K(m)}$ for some number field (or, $p$-adic field)~$K$.
The projectivity implies by definition that $f$ lifts to a homomorphism $\tilde{f}: \TT_\Z(S) \to \o_K$.
\end{proof}

We stress again that eigenforms mod~$p^m$ cannot, in general, be lifted to eigenforms if $m > 1$, but see Lemma~\ref{lem:ds}.

\subsection{Divided congruences}

In the next lemma we will show that when the coefficients are over a $\Q$-algebra $K$ one can split $S(K)$ into a direct sum according to weights. This does not hold true, in general, for arbitrary rings and leads to divided congruences.

\begin{lem}\label{lem:weights} Fix $M,b \in \N$ and let $S := \bigoplus_{k=1}^b S_k(\Gamma_1(M))$. Put $S_k := S_k(\Gamma_1(M))$ for each $k$.

If $K$ is any $\Q$-algebra, then one has $S(K) = \bigoplus_{k=1}^b S_k(K)$. Moreover, if $K$ is a field extension of~$\Q$ and $f \in S(K)$ is a normalized eigenform (say, it is an eigenform away from~$D$), then there is $k$ and a normalized eigenform $\tilde{f} \in S_k(L)$ for some finite extension $L/K$ such that $f(T_n) = \tilde{f}(T_n)$ for all $n$ coprime with~$D$.
\end{lem}

\begin{proof} For each $1 \le k \le b$, we have a natural homomorphism $\TT_\Q(S) \rightarrow \TT_\Q(S_k)$ given by restriction, and hence taking the product of these, we obtain an injective homomorphism $\TT_\Q(S) \rightarrow \prod_{k=1}^b \TT_\Q(S_k)$ of $\Q$-algebras. By Lemma~\ref{lem:intstr}, we have that
$$\dim_\Q \TT_\Q(S) = \dim_\C S = \sum_{k=1}^b \dim_\C S_k = \sum_{k=1}^b \dim_\Q \TT_\Q(S_k),$$
showing that $\TT_\Q(S) \cong \prod_{k=1}^b \TT_\Q(S_k)$.
Now, we see that
\begin{multline*}
      S(K)
=     \Hom_\Z(\TT_\Z(S),K)
\cong \Hom_\Q(\TT_\Z(S)\otimes_\Z\Q,K)
\cong \Hom_\Q(\TT_\Q(S),K)\\
\cong \Hom_\Q (\prod_{k=1}^b \TT_\Q(S_k),K)
\cong \bigoplus_{k=1}^b \Hom_\Z (\TT_\Z(S_k),K)
=     \bigoplus_{k=1}^b S_k(K).
\end{multline*}

Now assume that $K$ is a field extension of~$\Q$ and that $f \in S(K)$ is a normalized eigenform. By definition and Lemma~\ref{lem:eigenforms} this means that there is a positive integer~$D$ such that the restriction of~$f$ to $\TT_\Q^{(D)}(S)$ is a ring homomorphism $\TT_\Q^{(D)}(S) \to K$.
It can be extended to a ring homomorphism $\tilde{f}: \TT_\Q(S) \to L$ for some finite extension $L/K$, since in the integral extension of rings $\TT_\Q^{(D)}(S) \hookrightarrow \TT_\Q^{(1)}(S) = \TT_\Q(S)$ we need only choose a prime ideal of $\TT_\Q(S)$ lying over the prime ideal $\ker(f) \lhd \TT_\Q^{(D)}(S)$ (`going up', see \cite[Prop.\ 4.15]{eisenbud}). We need to make this extension in order to apply the results of the preceding paragraph, which we only proved for the full Hecke algebra.

To conclude, it suffices to note that every ring homomorphism $\TT_\Q(S) \to L$ factors through a unique $\TT_\Q(S_k)$. In order to see this, one can consider a complete set of orthogonal idempotents $e_1,\dots,e_n$ of $\TT_\Q(S)$, i.e.\ $e_i^2=e_i$, $e_ie_j=0$ for $i\neq j$ and $1=e_1+\dots+e_n$. As $L$ is a field and idempotents are mapped to idempotents, each $e_i$ is either mapped to~$0$ or~$1$, and as $0$ maps to $0$ and $1$ maps to $1$, there is precisely one idempotent that is mapped to~$1$, the others to~$0$. This establishes the final assertion.
\end{proof}

We explicitly point out the following easy consequence of Lemma~\ref{lem:weights}. We let $\o$ be the ring of integers of $K$, where $K$ is a number field or a finite extension of $\Q_p$. Consider again a space $S$ of the form $S = \bigoplus_{k=1}^b S_k(\Gamma_1(M))$.
By definition, it follows that we have
$$
S(\o) = \{f \in S(K) \;|\; f(T_n) \in \o \;\forall n\} = \{f \in \bigoplus_{k=1}^b S_k(K) \;|\; f(T_n) \in \o \;\forall n\}.
$$
Hence, our spaces $S^b(\Gamma_1(Np^r))(K)$ and $S^b(\Gamma_1(Np^r))(\o)$ are precisely the ones denoted $S^b(\Gamma_1(Np^r);K)$ and $S^b(\Gamma_1(Np^r);\o)$ on p.~550 of~\cite{hida1}.

Explicitly, $f \in S(\o)$ is of the form $f = \sum_k f_k$ with $f_k \in S_k(K)$, and although none of the $f_k$ need be in $S_k(\o)$, the sum has all its coefficients in $\o$. This is the origin of the name `divided congruence' for such an $f$: Suppose for example that we have forms $g_k\in S(\o)$ for various weights $k$ and that $\sum_k g_k \equiv 0 \pod{\pi^m}$ for some $m$ where $\pi$ is a uniformizer of $\o$. Putting $f_k := g_k/\pi^m$ for each $k$ we then have $f_k\in S_k(K)$ for all $k$ as well as $f:=\sum_k f_k \in S(\o)$. Conversely, any element of $S(\o)$ arises in this way by `dividing a congruence'.
\smallskip

We now turn our attention to the case $m=1$ (recall $\Zmod{p} = \Fbar_p$) and give a short proof of the Deligne--Serre lifting lemma (Lemme 6.11 of \cite{deligne-serre}) in terms of our setup. It implies that `dc-weak', `weak' and `strong' are equivalent notions for $m=1$.

\begin{lem}[Deligne--Serre lifting lemma]\label{lem:ds} Let $S = \bigoplus_{k=1}^b S_k(\Gamma_1(M))$ as above. Let $f \in S(\Fbar_p)$ be a dc-weak eigenform of level~$M$; say, it is an eigenform for all $T_n$ for $n$ coprime to some~$D \in \N$. Then there is a normalized holomorphic eigenform~$g$ of level~$M$ and some weight~$k$ such that $g(T_n) \equiv f(T_n) \pmod{p}$ for all $n$ coprime with~$D$ (i.e., $f$ is in fact a strong eigenform).
\end{lem}

\begin{proof}
The kernel of the ring homomorphism $f:\TT_\Z^{(D)}(S) \to \Fbar_p$ is a maximal ideal~$\fm$ of~$\TT_\Z^{(D)}(S)$.
Recall that $\TT_\Z^{(D)}(S)$ is free of finite rank as a $\Z$-module (by Lemma~\ref{lem:intstr}), whence it is equidimensional of Krull dimension~$1$, since $\Z_\ell \hookrightarrow \TT_\Z^{(D)}(S)_\lambda$ is an integral ring extension for any completion at a maximal ideal~$\lambda$ (say, lying above~$\ell$), and the Krull dimension is invariant under integral extensions (cf.\ \cite[Prop.\ 9.2]{eisenbud}, for instance.) Consequently, there is a prime ideal~$\p \lhd \TT_\Z^{(D)}(S)$ such that $\p \subsetneq \fm$. The quotient $\TT_\Z^{(D)}(S)/\p$ is an order in a number field~$K$. Moreover, there is a prime ideal $\q$ of $\o_K$ such that the following diagram is commutative:
$$ \xymatrix@=1cm{
\TT_\Z^{(D)}(S) \ar@{->>}[r] \ar@{=}[d] & \TT_\Z^{(D)}(S)/\p \ar@{->>}[d] \ar@{^(->}[r]
& \o_K  \ar@{->>}[d] \ar@{^(->}[r] & \C\\
\TT_\Z^{(D)}(S) \ar@{->>}[r] & \TT_\Z^{(D)}(S)/\fm \ar@{^(->}[r]
& \o_K/\q \ar@{^(->}[r] & \Fbar_p.}$$
The lower row is the ring homomorphism~$f$, and the upper row is a ring homomorphism that can be extended to a normalized Hecke eigenform~$g$ in $S(\C)=S$, which by (the proof of) Lemma~\ref{lem:weights} lies in $S_k(\Gamma_1(M))$ for some $k$. The commutativity of the diagram implies the claim on the reduction modulo~$p$.
\end{proof}

\section{Galois representations}\label{sec:galois}
In this section we construct a Galois representation attached to a dc-weak eigenform mod~$p^m$. For expressing its determinant, we find it convenient to work with Hida's stroke operator $\vphantom{1em}_{\mid \ell}$, which we denote $\stroke{\ell}$. We recall its definition from \cite{hida1}, p.~549. Let us consider again a space of the form $S = \bigoplus_{k=1}^b S_k(\Gamma_1(M))$ for some $b$. We now consider specifically a level $M$ written in the form
$$
M = Np^r
$$
where $p\nmid N$.

Let $Z = \Z_p^\times \times (\Z/N\Z)^\times$, into which we embed $\Z$ diagonally with dense image. We have a natural projection $\pi: Z \to \Z/p^r\Z \times \Z/N\Z \cong \Z/Np^r\Z$. Let first $f \in S$ be of weight~$k$. Hida defines for $z=(z_p,z_0) \in Z$
$$
[z]f = z_p^k \diam{\pi(z)} f
$$
where $\diam{\cdot}$ is the diamond operator. We recall that the diamond operator $\diam{d}$ for $d \in \Z/Np^r\Z$ is defined as $f_{|_k \sigma_d}$ with $\sigma_d \in \SL_2(\Z)$ such that $\sigma_d \equiv \abcd{d^{-1}} *0d \mod Np^r$. Since the diamond operator is multiplicative (it gives a group action of $\Z/Np^r\Z^\times$), so is the stroke operator.

We now show that for $z \in \Z$ the definition of $\stroke{z}$ can be made so as not to involve the weight. Let $\ell \nmid Np$ be a prime. Due to the well known equality
$$
\ell^{k-1} \diam{\ell} = T_\ell^2 - T_{\ell^2}
$$
(cf.\ for instance p.\ 53 of \cite{diamond-im}), one obtains
$$
\stroke{\ell} = \ell^k \diam{\ell} = \ell(T_\ell^2 - T_{\ell^2}).
$$
This first of all implies that $\stroke{\ell} \in \TT_\Z(S)$, since the right hand side clearly makes sense on~$S$ and is an element of~$\TT_\Z(S)$. Due to multiplicativity, all $\stroke{n}$ lie in $\TT_\Z(S)$ for $n \in \Z$. Consequently, $\stroke{n}$ acts on $S(A)$ for any ring~$A$ by its action via $\TT_\Z(S)$. Moreover, if $f \in S(A)$ is an eigenform for all $T_n$ ($n \in \N$), then it is also an eigenfunction for all~$\stroke{n}$. Strictly speaking it is not necessary for our purposes, but, nevertheless we mention that one can extend the stroke operator to a group action of~$Z$ on~$S(\o)$ for all complete $\Z_p$-algebras~$\o$ by continuity (which one must check). Thus, if $f \in S(\o)$ is an eigenfunction for all Hecke operators, then it is in particular an eigenfunction for all $\stroke{z}$ for $z \in Z$, whence sending $\stroke{z}$ to its eigenvalue on~$f$ gives rise to a character $\theta:Z \to \o^\times$, which we may also factor as $\theta = \eta \psi$ with $\psi : \Z/N\Z^\times \rightarrow \o^\times$ and $\eta : \Z_p^\times \rightarrow \o^\times$.

Since it is the starting point and the fundamental input to the sequel, we recall the existence theorem on $p$-adic Galois representations attached to normalized Hecke eigenforms for $k=2$ by Shimura, for $k>2$ by Deligne and for $k=1$ by Deligne and Serre (see, e.g., \cite{diamond-im}, p.~120). By $\Frob_\ell$ we always mean an arithmetic Frobenius element at~$\ell$.

\begin{thm}
\label{thm:galoisrepchar0} Suppose that $S = S_k(\Gamma_1(N p^r))$ with $k \ge 1$. Suppose $f \in S(\Qbar_p)$ is a normalized eigenform (say, it is an eigenform away from~$D$), so that $\diam{\ell} f = \chi(\ell) f$ for a character $\chi : (\Z/N p^r\Z)^\cross \rightarrow \Qbar_p^\cross$ for primes $\ell \nmid DNp$.

Then there is a continuous odd Galois representation
$$
\rho = \rho_{f,p}: \Gal(\Qbar/\Q) \to \GL_2(\Qbar_p)
$$
that is unramified outside $Np$ and satisfies
$$
\Tr(\rho(\Frob_\ell)) = f(T_\ell) \textrm{ and } \det(\rho(\Frob_\ell)) = \ell^{k-1} \chi(\ell)
$$
for all primes $\ell \nmid DNp$.
\end{thm}

\begin{cor}
\label{cor:galoisrepchar0dc} Suppose that $S = \bigoplus_{k=1}^b S_k(\Gamma_1(Np^r))$.
Suppose $f \in S(\Qbar_p)$ is a normalized eigenform (say, it is an eigenform away from~$D$), so that $[\ell] f = \eta(\ell) \psi(\ell) f$ for some characters $\psi : (\Z/N\Z)^\cross \rightarrow \Qbar_p^\cross$ and $\eta : \Z_p^\cross \rightarrow \Qbar_p^\cross$ for primes $\ell \nmid DNp$.

Then there is a continuous Galois representation
$$
\rho = \rho_{f,p}: \Gal(\Qbar/\Q) \to \GL_2(\Qbar_p)
$$
that is unramified outside $Np$ and satisfies
$$
\Tr(\rho(\Frob_\ell)) = f(T_\ell) \textrm{ and } \det(\rho(\Frob_\ell)) = f(\ell^{-1}[\ell]) = \ell^{-1} \eta(\ell) \psi(\ell)
$$
for all primes $\ell \nmid DNp$.
\end{cor}

\begin{proof}
From Lemma~\ref{lem:weights} we know that $f$ has a unique weight~$k$, i.e.\ lies in some $S_k(\Gamma_1(Np^r))(\Qbar_p)$. Thus, $f$ also gives rise to a character $\chi : \Z/Np^r\Z^\times \rightarrow \Qbar_p^\times$ by sending the diamond operator $\diam{\ell}$ to its eigenvalue on~$f$. The assertion now follows from the equation $\ell^k\diam{\ell} = \stroke{\ell}$ and Theorem~\ref{thm:galoisrepchar0}.
\end{proof}

\begin{cor}
\label{cor:galoisrepmodpdc} Suppose that $S = \bigoplus_{k=1}^b S_k(\Gamma_1(Np^r))$.
Suppose $\bar{f} \in S(\Fbar_p)$ is a normalized eigenform (say, it is an eigenform away from~$D$), so that $[\ell] \bar{f} = \eta(\ell) \psi(\ell) \bar{f}$ for some characters $\psi : (\Z/N\Z)^\cross \rightarrow \Fbar_p^\cross$ and $\eta : \Z_p^\cross \rightarrow \Fbar_p^\cross$ for primes $\ell \nmid DNp$.

Then there is a semisimple continuous Galois representation
$$
\rho = \rho_{\bar{f},p,1}: \Gal(\Qbar/\Q) \to \GL_2(\Fbar_p)
$$
that is unramified outside $Np$ and satisfies
$$
\Tr(\rho(\Frob_\ell)) = \bar{f}(T_\ell) \textrm{ and } \det(\rho(\Frob_\ell)) = \ell^{-1}\eta(\ell) \psi(\ell)
$$
for all primes $\ell \nmid DNp$.
\end{cor}

\begin{proof}
By Lemma~\ref{lem:ds}, there is an eigenform $f \in S(\Zbar_p)$ whose reduction is~$\bar{f}$, whence by Corollary~\ref{cor:galoisrepchar0dc} there is an attached Galois representation $\rho_{f,p}$. Due to the compactness of $\Gal(\Qbar/\Q)$ and the continuity, there is a finite extension $K/\Q_p$ such that the representation is isomorphic to one of the form $\Gal(\Qbar/\Q) \to \GL_2(\o_K)$. We define $\rho_{\bar{f},p,1}$ as the semisimplification of the reduction of this representation modulo the maximal ideal of~$\o_K$. It inherits the assertions on the characeristic polynomial at
$\Frob_\ell$ from $\rho_{f,p}$.
\end{proof}

Next we construct a Galois representation into the completed Hecke algebra.

\begin{thm}\label{thm:galoisrep} Suppose that $S = \bigoplus_{k=1}^b S_k(\Gamma_1(Np^r))$.

Let $D$ be a positive integer and let $\fm$ be a maximal ideal of $\hat{\TT}_\Z^{(D)}(S):= \TT_\Z^{(D)}(S) \otimes_\Z \Z_p$ and denote by $\hat{\TT}_\Z^{(D)}(S)_\fm$ the completion of $\hat{\TT}_\Z^{(D)}(S)$ at $\fm$. Assume that the residual Galois representation attached to
$$
\TT_\Z^{(D)}(S) \hookrightarrow \hat{\TT}_\Z^{(D)}(S) \to \hat{\TT}_\Z^{(D)}(S)_\fm \to \hat{\TT}_\Z^{(D)}(S)/\fm \hookrightarrow \Fbar_p
$$
is absolutely irreducible (note that this ring homomorphism can be extended to a normalized eigenform
in $S(\Fbar_p)$ by the argument using `going up' from the proof of Lemma~\ref{lem:weights}).

Then there is a continuous representation
$$
\rho = \rho_\fm: \Gal(\Qbar/\Q) \to \GL_2(\hat{\TT}_\Z^{(D)}(S)_\fm),
$$
that is unramified outside $Np$ and satisfies
$$
\Tr(\rho(\Frob_\ell)) = T_\ell \textrm{ and } \det(\rho(\Frob_\ell)) = \ell^{-1} [\ell]
$$
for all primes $\ell \nmid DNp$.
\end{thm}

\begin{proof}
Assume first that all prime divisors of $Np$ also divide~$D$. As the Hecke operators $T_n$ with $n$ coprime to~$D$ commute with each other and are diagonalizable (as elements of $\End_\C(S)$), there is a $\C$-basis $\Omega$ for $S$ consisting of eigenforms for $\TT_\Z^{(D)}(S)$. As $\TT_\Z^{(D)}(S)$ is finite over $\Z$, for each $f \in \Omega$, its image onto $\TT_\Z^{(D)}(\C f)$ is an order in a number field. Here, obviously $\TT_\Z^{(D)}(\C f)$ denotes the $\Z$-subalgebra of $\End_\C (\C f)$ generated by the $T_n$ with $(n,D)=1$.

Consider the natural map $\TT_\Z^{(D)}(S) \rightarrow \prod_{f \in \Omega} \TT_\Z^{(D)}(\C f)$, which is an injective homomorphism because $\Omega$ is a $\C$-basis for $S$. Letting $R = \TT_\Z^{(D)}(S) \otimes \Q$, we see that
$\prod_{f \in \Omega} \TT_\Z^{(D)}(\C f) \otimes \Q$ is a semisimple $R$-module, as each $\TT_\Z^{(D)}(\C f) \otimes \Q$ is a simple $R$-module. Thus, the $R$-submodule $R \subset \prod_{f \in \Omega} \TT_\Z^{(D)}(\C f) \otimes \Q$ is also a semisimple $R$-module, and $R = \TT_\Z^{(D)}(S) \otimes \Q$ is a semisimple ring. It follows that $\TT_\Z^{(D)}(S) \otimes \Q \cong \prod_i F_i$, where the $F_i$ are a finite collection of number fields. This means that $\TT_\Z^{(D)}(S) \otimes \Q_p \cong \prod_i K_i$ with the $K_i$ a finite collection of finite extensions of $\Q_p$.

Thus, there is an injective homomorphism $\hat{\TT}_\Z^{(D)}(S) \hookrightarrow \prod_i \o_i$, where $\o_i$ is the ring of integers of $K_i$. Hence, there is an injective homomorphism $\hat{\TT}_\Z^{(D)}(S)_{\fm} \hookrightarrow \prod_i \o_i$, which is obtained from the previous one by discarding factors where $\fm$ is not sent into the maximal ideal of $\o_i$. Each projection $\hat{\TT}_\Z^{(D)}(S)_\fm \rightarrow \o_i$ is a map of local rings.

Each ring homomorphism $g_i : \TT_\Z^{(D)}(S) \rightarrow K_i$ lifts to a ring homomorphism $f_i : \TT_\Z(S) \rightarrow E_i$, where $E_i$ is a finite extension of $K_i$. By Corollary~\ref{cor:galoisrepchar0dc}, for each $i$, there is a continuous Galois representation $\rho_i : \Gal(\Qbar/\Q) \rightarrow \GL_2(\o_i')$, where $\o_i'$ is the ring of integers of $E_i$.

Let $\rho = \prod_i \rho_i : \Gal(\Qbar/\Q) \rightarrow \prod_i \GL_2(\o_i') = \GL_2(\prod_i \o'_i)$ be the product representation. Under the inclusion $\TT_\Z^{(D)}(S) \hookrightarrow \prod_i \o'_i$, we see for $\ell \nmid DNp$, that $\Tr \rho(\Frob_\ell) = T_\ell$ and $\det \rho(\Frob_\ell) = \ell^{-1} [\ell]$. The residual Galois representations $\overline{\rho}_i : \Gal(\Qbar/\Q) \rightarrow \GL_2(k_i')$, where $k_i'$ is the residue field of $\o_i'$, are all isomorphic to the Galois representation attached to $\TT_\Z^{(D)}(S) \rightarrow \hat{\TT}_\Z^{(D)}(S)/\fm$, and hence are absolutely irreducible.
\smallskip

We will now apply \cite[Th\'eor\`eme 2]{carayol}, the setting of which is as follows. Suppose that $A$ is a local henselian ring with maximal ideal $\fm$ and residue field $F = A/\fm$. Assume that the Brauer group of $F$ vanishes (this will be the case if $F$ is finite). Let $A'$ be a semilocal extension of $A$: $A' = \prod_i A_i'$ with the $A_i'$ local rings with maximal ideals $\fm_i$ and residue fields $F_i' = A_i'/\fm_i$ that are extensions of $F$. Suppose that we are given an $A$-algebra $R$ and a representation $\rho' = \prod_i \rho_i' : ~ R\otimes_A A' \rightarrow M_n(A') = \prod_i M_n(A_i')$. Suppose further that $\Tr \rho'(r\otimes 1) \in A$ for all $r\in R$, and that the attached residual representations $\bar\rho'_i : ~ R\otimes_A F_i' \rightarrow M_n(F_i')$ are all absolutely irreducible. The conclusion of \cite[Th\'eor\`eme 2]{carayol} is then that $\rho'$ is (equivalent to) the base change from $A$ to $A'$ of a representation $\rho : ~ R \rightarrow M_n(A)$.
\smallskip

We see that we can apply \cite[Th\'eor\`eme 2]{carayol} with $A = \hat{\TT}_\Z^{(D)}(S)_{\fm}$ (which is a complete local ring, hence henselian), and $A' = \prod_i \o_i'$ (which is a semilocal extension of $A$), to deduce that the representation $\rho$ descends to a continuous Galois representation
\begin{equation*}
  \rho_\fm : \Gal(\Qbar/\Q) \rightarrow \GL_2(\hat{\TT}_\Z^{(D)}(S)_{\fm}),
\end{equation*}
as claimed.

For the general case, when $D$ is not divisible by all prime divisors of $Np$, one first applies the above with $D' := DNp$ and the maximal ideal
$\fm'$ of $\hat{\TT}_\Z^{(D')}$ given as $\fm \cap \hat{\TT}_\Z^{(D')}$ to obtain $\rho_{\fm'} : \Gal(\Qbar/\Q) \rightarrow \GL_2(\hat{\TT}_\Z^{(D')}(S)_{\fm'})$, which can finally be composed with the natural map $\hat{\TT}_\Z^{(D')}(S)_{\fm'} \to \hat{\TT}_\Z^{(D)}(S)_{\fm}$.
\end{proof}

\begin{cor}\label{cor:galoisrepAdc} Suppose that $S = \bigoplus_{k=1}^b S_k(\Gamma_1(Np^r))$.
Let $A$ be a complete local ring with maximal ideal $\p$ of residue characteristic~$p$. Suppose $f \in S(A)$ is a normalized eigenform (say, it is an eigenform away from~$D$), so that $[\ell] f = \eta(\ell) \psi(\ell) f$ for some characters $\psi : (\Z/N\Z)^\cross \rightarrow A^\cross$ and $\eta : \Z_p^\cross \rightarrow A^\cross$, for all $\ell \nmid DNp$.

Assume the Galois representation attached to the reduction $\bar{f} : \TT_\Z(S) \rightarrow A \rightarrow A/\p$ mod $\p$ of $f$, which defines a normalized eigenform in $S(\Fbar_p)$, is absolutely irreducible (cf.\ Corollary~\ref{cor:galoisrepmodpdc}).

Then there is a continuous Galois representation
$$
\rho = \rho_{f,p}: \Gal(\Qbar/\Q) \to \GL_2(A)
$$
that is unramified outside $Np$ and satisfies
$$
\Tr(\rho(\Frob_\ell)) = f(T_\ell) \textrm{ and } \det(\rho(\Frob_\ell)) = \ell^{-1} \eta(\ell) \psi(\ell)
$$
for all primes $\ell \nmid DNp$.
\end{cor}

\begin{proof}
Since $S(A)$ is a normalized eigenform, $f : \TT_\Z^{(D)}(S) \rightarrow A$ is a ring homomorphism, which factors through $\hat{\TT}_\Z^{(D)}(S)_\fm$ for some maximal ideal~$\fm$, since $A$ is complete and local. (The ideal~$\fm$ can be seen as the kernel of $\hat{\TT}_\Z^{(D)}(S)  \to A \twoheadrightarrow A/\p$.) We thus have a ring homomorphism $\hat{\TT}_\Z^{(D)}(S)_\fm \rightarrow A$. Composing this with the Galois representation $\rho_\fm$ from Theorem~\ref{thm:galoisrep} yields the desired Galois representation~$\rho_{f,p}$.
\end{proof}

For our applications, the most important case of Corollary \ref{cor:galoisrepAdc} is when $A = \Zmod{p^m}$. Note also that one can attach a Galois representation by the above corollary to any ring homomorphism $\TT_\Z^{(D)}(S) \rightarrow A$.

\begin{proof}[Proof of Theorem~\ref{main-galois}]
It suffices to apply Corollary~\ref{cor:galoisrepAdc} with $A = \Zmod{p^m}$.
\end{proof}

\section{Stripping powers of \texorpdfstring{$p$}{p} from the level}\label{sec:stripping}

In this section, we use results of Katz and Hida in order to remove powers of $p$ from the level of cusp forms over $\Zmod{p^m}$, at the expense of using divided congruences.

Let $M$ be any positive integer. Let $\o$ be the ring of integers of either a number field or a finite extension of~$\Q_p$. Define
$$
{\mathbb S}(\Gamma_1(M))(\o) := \ilim_{b\ge 1} S^b(\Gamma_1(M))(\o).
$$
Specializing $\o$ to $\Z_p$, we complete these spaces, which are of infinite rank, and put
\begin{align*}
\bar{\mathbb S}(\Gamma_1(M);\Z_p)
& :=    \prlim_{m \ge 1} {\mathbb S}(\Gamma_1(M))(\Z_p) \otimes_{\Z_p} \Z_p/p^m\Z_p \\
& \cong \prlim_{m \ge 1} {\mathbb S}(\Gamma_1(M))(\Z_p/p^m\Z_p) \\
& \cong \prlim_{m \ge 1} {\mathbb S}(\Gamma_1(M))(\Z) \otimes_\Z \Z/p^m\Z.
\end{align*}
For the isomorphisms, we use that the direct limit is exact on modules and
\begin{multline}\label{eq:compare}
S^b(\Gamma_1(M))(\Z_p) \otimes_{\Z_p} \Z_p/p^m\Z_p \cong
S^b(\Gamma_1(M))(\Z_p/p^m\Z_p)\\
\cong S^b(\Gamma_1(M))(\Z)\otimes_{\Z} \Z/p^m\Z,
\end{multline}
which is an application of Lemma~\ref{lem:basechange}.

\begin{thm}[Hida]\label{thm:hida}
Assume $p \ge 5$. Let $N$ be a positive integer prime to~$p$ and let $r \in \N$.

Under the $q$-expansion map, the images of $\bar{\mathbb S}(\Gamma_1(N);\Z_p)$ and $\bar{\mathbb S}(\Gamma_1(Np^r);\Z_p)$ agree inside $\Z_p[[q]]$.
\end{thm}

\begin{proof}
This result is stated in \cite[(1.3)]{hida1}. (For the sake of completeness, let us point out that Hida's spaces $\bar{\mathbb S}$ result from completing with respect to a natural norm $|f|_p$ whereas our $\bar{\mathbb S}$ arise as projective limits. However, it is easy to see that the two points of view are equivalent. Also, the spaces of modular forms that both we and Hida are working with are spaces of `classical forms', cf.\ Section~\ref{sec:mf}.)
\end{proof}

\begin{prop}\label{prop:remove}
Assume $p \ge 5$. Let $K$ be a number field and $\o$ its ring of integers. Let $N$ be a positive integer prime to~$p$ and $r \in \N$. Let $f \in {\mathbb S}(\Gamma_1(Np^r);\o)$.

Then for all $m \in \N$, there are $b_m \ge 1$ and $g_m \in S^{b_m}(\Gamma_1(N))(\o) \hookrightarrow {\mathbb S}(\Gamma_1(N))(\o)$ such that $f(q) \equiv g_m(q) \pmod{p^m}$.
\end{prop}

\begin{proof}
There is a $b \ge 1$ such that $f \in S^b(\Gamma_1(Np^r))(\o)$. Since $S^{b}(\Gamma_1(Np^r))(\o) = S^{b}(\Gamma_1(Np^r))(\Z) \otimes_\Z \o$ by Lemma~\ref{lem:basechange}, there are $f_i \in S^{b}(\Gamma_1(Np^r))(\Z)$ and $a_i \in \o$ for $i=1,\dots,t$ such that $f = \sum_{i=1}^t a_i f_i$.

There is a $b_m \ge 1$ such that the images of $f_i$ under the composition of maps
\begin{multline*}
S^{b}(\Gamma_1(Np^r))(\Z) \hookrightarrow {\mathbb S}(\Gamma_1(Np^r))(\Z) \hookrightarrow \bar{\mathbb S}(\Gamma_1(Np^r);\Z_p) \xrightarrow{\sim \textnormal{Thm.~\ref{thm:hida}}} \bar{\mathbb S}(\Gamma_1(N);\Z_p)\\
\twoheadrightarrow \bar{\mathbb S}(\Gamma_1(N);\Z_p) \otimes_{\Z_p} \Z_p/p^m\Z_p \cong {\mathbb S}(\Gamma_1(N))(\Z_p) \otimes_{\Z_p} \Z_p/p^m\Z_p
\end{multline*}
all lie in $S^{b_m}(\Gamma_1(N))(\Z_p) \otimes_{\Z_p} \Z_p/p^m\Z_p \cong S^{b_m}(\Gamma_1(N))(\Z) \otimes \Z/p^m\Z$. Hence, there are $g_{i,m} \in S^{b_m}(\Gamma_1(N))(\Z)$ such that $f_i(q) \equiv g_{i,m}(q) \pmod{p^m}$. Finally, $g_m := \sum_{i=1}^t a_i g_{i,m} \in S^{b_m}(\Gamma_1(N))(\o)$ is such that $f(q) \equiv g_m(q) \pmod{p^m}$.
\end{proof}

\begin{proof}[Proof of Proposition~\ref{prop:hatada}] By virtue of Lemma~\ref{lem:lifting}, we can lift $f$ to an element of $S^b(\Gamma_1(Np^r))(\o_K)$ for some number field~$K$. Hence, Proposition~\ref{prop:remove} applies, yielding a form~$g_m$, whose reduction mod $p^m$ satisfies the requirements.
\end{proof}

\section{On the relationships between strong, weak and dc-weak}\label{relations} In this section we make a number of remarks concerning the notions of a mod $p^m$ Galois representation arising `strongly', `weakly', and `dc-weakly' from $\Gamma_1(M)$, and the accompanying notions of 'strong', 'weak', and 'dc-weak' eigenforms. Notice that Lemma \ref{lem:ds} above implies that these are equivalent notions when $m=1$. We show here that the three notions are not equivalent in general (for fixed level~$M$), but must leave as an open question to classify the conditions under which the three notions coincide.

We also give a few illustrative examples at the end of the section.

\subsection{Nebentypus obstructions}

We show here that in order to strip powers of $p$ from the level of a Galois representation which is strongly modular, it is necessary in general to consider the Galois representations attached to dc-weak eigenforms. The argument uses certain nebentypus obstructions that also -- in general -- prohibit `weak' eigenforms of level prime-to-$p$ from coinciding with `dc-weak' eigenforms. The argument can also be seen as a demonstration of the fact that the nebentypus restriction in Hatada's theorem (Theorem \ref{thm:hatada} in the introduction) are in fact necessary.
\smallskip

Assume $p \nmid N$ and let $f_0 \in S_k(\Gamma_1(Np^r))(\Zmod{p^m})$ be a strong eigenform. A consequence of Theorem~\ref{thm:main} is that the Galois representation $\rho_{f_0,p,m}$ dc-weakly arises from $\Gamma_1(N)$. We show that $\rho_{f_0,p,m}$ does not, in general, weakly arise from $\Gamma_1(N)$.

By the definition of `strong eigenform' there is $f \in S_k(\Gamma_1(Np^r))(\Zbar_p)$ that is normalized and an eigenform outside some positive integer $D$ and that reduces to $f_0$. Suppose that $\langle \ell \rangle f = \chi(\ell) f$ for primes $\ell$ with $\ell \nmid DNp$, with a character $\chi$ that we decompose as $\chi = \psi \omega^i \eta$ where $\psi$ is a character of conductor dividing $N$, $\omega$ is the Teichm\"uller lift of the mod~$p$ cyclotomic character, and $\eta$ is a character of conductor dividing $p^r$ and order a power of $p$. Assume $p$ is odd, $r \ge 2$, $\eta \not= 1$, and $m \ge 2$. Let $\rho_{f,p,m}$ be the mod $p^m$ representation attached to $f$. Then $\rho_{f,p,m} \cong \rho_{f_0,p,m}$. By the argument below, it is not possible to find a weak eigenform $g \in S_{k'}(\Gamma_1(N))(\Zmod{p^m})$ of any weight $k'$ such that $\rho_{g,p,m} \cong \rho_{f,p,m}$.

Let $\eta$ have order $p^s$ where $1 \le s \le r-1$. Then we may regard $\eta$ as a character $\eta : (\Z/p^r \Z)^\cross \rightarrow \Z_p[\zeta]^\cross$, where $\zeta$ is a primitive $p^s$-th root of unity. Assume there exists a weak eigenform $g$ on $\Gamma_1(N)$ such that $\rho_{f,p,m} \cong \rho_{g,p,m}$. As $g$ is an eigenform for $\langle \ell \rangle$ for primes $\ell$ with $\ell \nmid DNp$, we have that $\langle \ell \rangle g = \psi'(\ell) g$, where
$$
\psi' : (\Z/N\Z)^\cross \rightarrow \overline{\Z/p^m\Z}^\cross
$$
is a mod $p^m$ character of conductor dividing $N$. Since $\rho_{f,p,m} \cong \rho_{g,p,m}$, we have that $\det \rho_{g,p,m} = \det \rho_{f,p,m}$. Now, we know that
$$
\det \rho_{f,p,m} \equiv \epsilon^{k-1} \psi \omega^i \eta \pmod{p^m} ,
$$
with $\epsilon$ the $p$-adic cyclotomic character. Also, from the construction of the Galois representation attached to $g$ (cf.\ \cite[Th\'eo\-r\`eme 3]{carayol}), we have that
$$
\det \rho_{g,p,m} \equiv \epsilon^{k'-1} \psi' \pmod{p^m} .
$$

Hence, after restricting to the inertia group at $p$, we have that
$$
\epsilon^{k'-1} \equiv \eta \epsilon^{k-1} \pmod{p^m}
$$
as characters of $\Z_p^\cross$, or equivalently $\eta \equiv \epsilon^{k'-k} \pmod {p^m}$.

The cyclotomic character $\epsilon(x) = x$ has values in $\Z_p$, however the image of the character $\eta$ in $\Z_p[\zeta]$ contains~$\zeta$. Since $m \ge 2$, the injection $$\Z_p/(p^m) \hookrightarrow \Z_p[\zeta]/(1-\zeta)^{(m-1)p^{s-1}(p-1)+1}$$ is not a surjection. Thus, the reduction $\bmod{\, p^m}$ of~$\epsilon^{k'-k}$ has values in $\Z_p/(p^m)$, but the reduction
$\bmod{\, p^m}$ of~$\eta$ does not. This contradicts the equality $\eta \equiv \epsilon^{k'-k} \pmod{p^m}$.

Note for $m=1$, we always have $\eta \equiv 1 \pmod p$ and hence it is possible to have the equality of characters in this situation.

Although the main purpose of this section is to show that there exist $\rho_{f,p,m}$ which arise strongly from $\Gamma_1(Np^r)$ and do not arise weakly from $\Gamma_1(N)$, we note the proof shows there exist dc-weak eigenforms of level $N$ which are not weak eigenforms of level $N$.

\subsection{On the weights in divided congruences} In this subsection we show that under {\it certain} conditions, the weights occurring in a dc-weak eigenform satis\-fy enough congruence conditions so that one can equalize them using suitable powers of Eisenstein series, a technique which was used in \cite{ckr}. In fact, Corollary \ref{weights_cong_new} below is a generalization of some of the results in \cite{ckr}, using different methods. We impose here that $p>2$.

\begin{lem}\label{lem:independence}
Let $\o$ be a local ring with maximal ideal $\p$, and let $M$ be a finite projective $\o$-module. If $\bar{f}_1, \ldots, \bar{f}_n \in M /\p M$ are linearly independent over $\o/\p$, then $f_1, \ldots, f_n \in M/\p^m M$ are linearly independent over $\o/\p^m$.
\end{lem}

\begin{proof} By \cite[Chap.\ X, Theorem 4.4]{L}, we have that $M$ is isomorphic to $F \oplus \bigoplus_{i=1}^n \o f_i$ with $F$ a free $\o$-module, from which the assertion immediately follows.
\end{proof}

\begin{prop}\label{prop:weights_cong} Let $\OO$ be the ring of integers of a finite extension of $\Q_p$. Let $f_i \in S_{k_i}(\Gamma_1(Np^r))(\OO)$ for $i=1,\dots,t$, where the $k_i$ are distinct, and suppose $[\ell] f_i = \ell^{k_i} \psi_i(\ell) \eta_i(\ell) f_i$, for $\ell \nmid DNp$ (for some positive integer~$D$), where $\psi_i : \Z/N\Z^\cross \rightarrow \OO^\cross$, $\eta_i : \Z/p^r\Z^\cross \rightarrow \OO^\cross$ have finite order. Suppose also that the $q$-expansions $f_i(q) \pmod{p}$ are linearly independent over $\overline{\Z/p\Z} = \overline{\F}_p$.

Put $f:=\sum_{i=1}^t f_i$ and assume that $f$ is an eigenform for the operators $[\ell]$ (e.g. this is the case if $f$ is a dc-weak eigenform).

Then $k_1 \equiv k_2 \equiv \dots \equiv k_t \pmod{\varphi(p^m)/h}$, where $\varphi$ is the Euler-$\varphi$-function, and $h$ is the least common multiple of the orders of the $\eta_i \pmod{p^m}$.
\end{prop}
\begin{proof} Denote by $\lambda$, $\lambda_i$ the $[\ell]$-eigenvalue of $f$ and the $f_i$, respectively. Then we have $\lambda f \equiv \sum_{i=1}^t \lambda_i f_i(q) \pmod{p^m}$, whence $\sum_{i=1}^t (\lambda - \lambda_i) f_i(q) \equiv 0 \pmod{p^m}$. Lemma \ref{lem:independence} applied with $M = \OO[[q]]/(q^L)$ for suitable $L$ large enough (for instance, take $L$ so that the $q$-expansion map $\oplus_{i=1}^t S_{k_i}(\Gamma_1(Np^r))(\OO) \rightarrow \OO[[q]]/(q^L)$ is injective), shows that $\lambda \equiv \lambda_i \pmod{p^m}$ for all $i$. In particular, we have $\lambda_i \equiv \lambda_j \pmod{p^m}$ for all $i,j$.

We have $\lambda_i = \ell^{k_i} \psi_i(\ell) \eta_i(\ell)$. If $\ell \equiv 1 \pmod{N}$ then $\psi_i(\ell) = 1$. For such $\ell$ we thus have
$$
\ell^{k_i h} = \lambda_i^h \equiv \lambda_j^h = \ell^{k_j h} \pmod{p^m}
$$
for all $i,j$, by the definition of $h$.

By Chebotarev's density theorem, we can choose $\ell$ so that in addition to the property $\ell \equiv 1 \pmod{N}$, we have that $\ell$ is a generator of $(\Z/p^m\Z)^\cross$ (here we use that $p$ is odd and that $p\nmid N$.) It then follows that $k_1 h \equiv k_2 h \equiv \ldots \equiv k_t h \pmod{\varphi(p^m)}$ as desired.
\end{proof}

The proposition has the following application. Suppose that $f$ is a dc-weak eigenform mod $p^m$ at level $N$ of the form $f=\sum_{i=1}^t f_i$ with $f_i \in S_{k_i}(\Gamma_1(N))(\OO)$ for $i=1,\dots,t$, where the $k_i$ are distinct. Suppose that each $f_i$ has a nebentypus and that, crucially, the $q$-expansions $f_i(q) \pmod{p}$ are linearly independent over $\overline{\F}_p$.

Then the proposition applies with $h=1$ and shows that we have $k_1 \equiv \dots \equiv k_t \pmod{\varphi(p^m)}$. Without loss of generality suppose that $k_t$ is the largest of the weights. When $p\ge 5$, we can use, as in \cite{ckr}, the Eisenstein series $E := E_{p-1}$ of weight $p-1$ and level $1$, normalized in the usual way so that its $q$-expansion is congruent to $1 \pmod p$. The form $\tilde{E} := E^{p^{m-1}}$ is of weight $\varphi(p^m) = (p-1)p^{m-1}$, level $1$, and is congruent to $1 \pmod{p^m}$. Due to the congruence on the weights, we may multiply each $f_i$ for $i=1,\dots, t-1$ with a suitable power of $\tilde{E}$ so as to make it into a form of weight $k_t$ with the same $q$-expansion mod $p^m$. Consequently, in weight $k_t$ and level $N$ there is a form that is congruent to $f$ mod $p^m$, i.e., $f$ is in fact a weak eigenform mod $p^m$ at level $N$.
\smallskip

We also record the following variant of Proposition \ref{prop:weights_cong} as it represents a generalization of some of the results of \cite{ckr}.

\begin{cor}\label{weights_cong_new} Let $\OO$ be the ring of integers of a finite extension of $\Q_p$. Let $f_i \in S_{k_i}(\Gamma_1(Np^r))(\OO)$ for $i=1,\dots,t$ satisfy $f_1(q) + \ldots + f_t(q) \equiv 0 \pmod{p^m}$, where the $k_i$ are distinct, and suppose $[\ell] f_i = \ell^{k_i} \psi_i(\ell) \eta_i(\ell) f_i$, for $\ell \nmid DNp$ (for some positive integer~$D$), where $\psi_i : \Z/N\Z^\cross \rightarrow \OO^\cross$, $\eta_i : \Z/p^r\Z^\cross \rightarrow \OO^\cross$ have finite order. Suppose for some $i$, the $q$-expansions $f_j(q) \pmod{p}$, $j \not= i$ are linearly independent over $\overline{\Z/p\Z} = \overline{\F}_p$.

Then $k_1 \equiv k_2 \equiv \dots \equiv k_t \pmod{\varphi(p^m)/h}$, where $\varphi$ is the Euler-$\varphi$-function, and $h$ is the least common multiple of the orders of the $\eta_j \pmod{p^m}$.
\end{cor}
\begin{proof} Without loss of generality, assume $i = 1$. As $-f_1(q) \equiv \sum_{i=2}^t f_i \pmod{p^m}$ the proof of Proposition \ref{prop:weights_cong} shows that we have
$$
\ell^{k_1} \psi_1(\ell) \eta_1(\ell) \equiv \ell^{k_i} \psi_i(\ell) \eta_i(\ell) \pmod{p^m}
$$
for $i=2,\ldots,t$, and the desired congruences then follow in the same way.
\end{proof}

\subsection{Examples}
We do not know whether the notions `strong' and `weak' at a fixed level prime-to-$p$ coincide in general when the weight is allowed to vary. However, the following examples seem to suggest that they do not. It is of obvious interest to resolve this question, perhaps first by looking for a numerical counterexample. (There would be theoretical problems to consider before one could do that, specifically obtaining a weight bound.)
\smallskip

In \cite{taixes-wiese}, Section 4.2, one has an example mod~$9$ in weight~$2$ for $\Gamma_0(71)$. Note: the notion of strong and weak eigenform in loc.\ cit.\ is at a single weight, i.e.\ the weight is also fixed, which differs from our terminology. The example in loc.\ cit.\ shows that there is a cusp form in $S_2(\Gamma_0(71))$ which when reduced mod $9$ is an eigenform mod $9$, and which does not arise from the reduction mod $9$ of an eigenform (for all $T_n$) on $S_2(\Gamma_0(71))$.
\smallskip

As a general mechanism for producing eigenvalues mod $p^2$ that do not lift to characteristic $0$ (at the same weight), consider the following setup: Let $p$ be an odd prime. Suppose $f, g \in M = \Z_p^2$ and $f \equiv g \not\equiv 0 \pmod{pM}$. Suppose that $T$ is an endomorphism of $M$, that $Tf = \lambda f$, that $Tg = \mu g$, and that $\left\{ f, g\right\}$ is a basis for $M \otimes \overline{\Q}_p$. Then $\lambda \equiv \mu \pmod p$. Suppose further that $\lambda \not\equiv \mu \pmod {p^2}$. Consider $h = f + g$. Then $T h - \frac{\lambda + \mu}{2} h = \lambda f + \mu g - \frac{\lambda + \mu}{2} h = \frac{\lambda - \mu}{2} (f - g)$, so we have $T h \equiv \frac{\lambda + \mu}{2} h \pmod {p^2 M}$. Thus, $\frac{\lambda + \mu}{2} \in \Z_p/p^2\Z_p$ is an eigenvalue which does not lift to $\overline{\Z}_p$ as an eigenvalue of $T$ acting on $M \otimes \overline{\Q}_p$.
\smallskip

Using MAGMA, cf.\ \cite{Magma}, we found the following concrete example involving modular forms of the same weight. Let $S$ be the free $\Z_3$-module of rank $5$ which is obtained from the image of $S_2(\Gamma_0(52))(\Z_3)$ in $\Z_3[[q]]$ under the $q$-expansion map. Consider
\medskip
\begin{verbatim}
f = q + 2*q^5 - 2*q^7 - 3*q^9 - 2*q^11 - q^13 + 6*q^17 - 6*q^19
+ O(q^20)

g = q + q^2 - 3*q^3 + q^4 - q^5 - 3*q^6 + q^7 + q^8 + 6*q^9 - q^10
- 2*q^11 - 3*q^12 - q^13 + q^14 + 3*q^15 + q^16 - 3*q^17 + 6*q^18
+ 6*q^19 + O(q^20)
\end{verbatim}
\medskip
which are the $q$-expansions, respectively, of newform $1$ in $S_2(\Gamma_0(52))$, and newform $2$ in $S_2(\Gamma_0(26))$, in MAGMA's internal labeling system.

By Lemma 4.6.5 of \cite{miyake} the series $\sum_{2\nmid n} a_n(g)$ is the $q$-expansion of an element $\tilde{g}$ of $S_2(\Gamma_0(52))$. Also, $\tilde{g}$ is an eigenform for all the Hecke operators $T_n$ for all $n \ge 1$ (this can be checked from the explicit formulae of $T_n$ acting on $q$-expansions for $n$ prime).

We have that $\left\{ f, \tilde{g} \right\}$ is $\overline{\Q}_3$-linearly independent, and that $a_n(f) \equiv a_n(\tilde{g}) \pmod 3$ for all $n$ (to see that the congruence holds for all $n$ use the Sturm bound, cf.\ Theorem 1 of \cite{sturm}. The bound is $14$ in this case.)

Let $h = \frac{f+\tilde{g}}{2}$ so that
\medskip
\begin{verbatim}
h = q - 3/2*q^3 + 1/2*q^5 - 1/2*q^7 + 3/2*q^9 - 2*q^11
+ 3/2*q^15 + 3/2*q^17 + O(q^20)
\end{verbatim}
\medskip

By the arguments above, $(h \bmod 9)$ is a $\Z_3/9\Z_3$-eigenform for the Hecke operators $T_n$ for all $n \ge 1$. Furthermore, the system of eigenvalues for $T_n$ for all $n \ge 1$ does not arise from the reductions mod $9$ of $f$, nor $g$, as well as $g_1$, the other newform in $S_2(\Gamma_0(26))$. Thus, we see finally that the system of eigenvalues $\in \Z_3/9\Z_3$ for the $T_n$ for all $n \ge 1$, acting on $h$, do not lift to $\overline{\Z}_3$ inside $S \otimes \overline{\Q}_3$.

On the other hand, we make the remark that we have attached a mod $9$ Galois representation to $h$ by Corollary~\ref{cor:galoisrepAdc} (note: the residual mod $3$ representation is absolutely irreducible).


\end{document}